\documentclass[12pt]{article}
\usepackage{amsmath,mathtools,amsfonts,amssymb,amsthm,bbm,color,graphicx}
\usepackage[linesnumbered,ruled,lined]{algorithm2e}

\usepackage[active]{srcltx}
\usepackage[utf8]{inputenc}

\DeclareMathOperator{\dist}{dist}

\DeclareMathOperator{\argmin}{argmin}
\DeclareMathOperator{\argmax}{argmax}
\DeclareMathOperator{\proj}{proj}

\begin{document}
\newtheorem{oberklasse}{OberKlasse}
\newtheorem{definition}[oberklasse]{Definition}
\newtheorem{lemma}[oberklasse]{Lemma}
\newtheorem{proposition}[oberklasse]{Proposition}
\newtheorem{theorem}[oberklasse]{Theorem}
\newtheorem{corollary}[oberklasse]{Corollary}
\newtheorem{remark}[oberklasse]{Remark}
\newtheorem{example}[oberklasse]{Example}

\newcommand{\R}{\mathbbm{R}}
\newcommand{\N}{\mathbbm{N}}
\newcommand{\Z}{\mathbbm{Z}}
\newcommand{\mc}{\mathcal}
\newcommand{\eps}{\varepsilon}
\renewcommand{\phi}{\varphi}

\allowdisplaybreaks[1] 

\title{A learning-enhanced projection method for solving convex feasibility 
problems}
\author{Janosch Rieger}
\date{\today}
\maketitle

\begin{abstract}
We propose a generalization of the method of cyclic projections,
which uses the lengths of projection steps carried out in the past 
to learn about the geometry of the problem and decides on this basis
which projections to carry out in the future.
We prove the convergence of this algorithm and illustrate its
behavior in a first numerical study.
\end{abstract}

\medskip



\section{Introduction}

The method of cyclic projections, originally proposed in \cite{Bregman},
is an established numerical algorithm, which computes a point in the
intersection of finitely many closed convex subsets of a Hilbert space 
when this intersection is nonempty.
A broad overview over convergence properties of this method as well as the
underlying theory is given in \cite{BB}, \cite{Bauschke:Combettes},  
\cite{Deutsch} and the references therein.

\medskip

Estimates for the speed of convergence of the method of cyclic projections
are well-known in the case when the sets are affine linear subspaces.
For this situation, accelerated variants of the original scheme, which are 
often based on line-search ideas, have been developed, see e.g.\ 
\cite{Bauschke} and \cite{Gearhart}.
Recently, a first result on the speed of convergence of the method of cyclic 
projections has been given in the case of semi-algebraic sets, see \cite{Guoyin}.
In general, however, the method can be arbitrarily slow, 
see \cite{Franchetti} for a pathological example.

\medskip

When the sets are affine linear subspaces with codimension 1, 
the method of cyclic projections reduces to the Kaczmarz method, 
see \cite{Kaczmarz}, which has gained popularity in the context of very large, 
but sparse consistent linear systems, see \cite{Censor}.
A probabilistic version of this algorithm, which converges exponentially 
in expectation, has been introduced in \cite{Strohmer}, and an accelerated
version of this method has been proposed in \cite{Bai}.
The Kaczmarz method is frequently used in medical imaging, see \cite{Jiang},
where block and column action strategies have become a topic of
active research interest \cite{Li} and \cite{Elfving}.

\medskip

The numerical method presented in this paper is supposed to accelerate
the method of cyclic projections in settings where the above-mentioned 
refined algorithms for subspaces are not applicable.
The guiding idea behind the method is to gather as much information 
on the relative geometry of the closed convex sets from the lengths of the
projection steps carried out in the past.
This is motivated by the convergence proof in \cite{Bregman}, which
reveals that the performance of the algorithm is in worst case determined 
by the lengths of the projection steps carried out.

\medskip

We prove that our method converges, using techniques which are 
common in the dynamical systems community.
The main challenge is to guarantee convergence for a reasonably 
broad class of strategies our basic algorithm can be equipped with.
As it seems very hard to quantify a speed of convergence even in the 
subspace case, we provide several numerical studies performed on a toy
example, which provide some insight as to why and how our method can
outperform the standard methods of cyclic and random projections.

\section{The algorithm} \label{sec2}

Given closed convex sets $C_1,\ldots,C_N\subset\R^d$ with 
$C:=\cap_{j=1}^NC_j\neq\emptyset$, we wish to find a point $x^*\in C$.
We first present two common projection algorithms for solving 
this problem in Section \ref{sec:benchmark}.
Then we propose a new projection algorithm in Section \ref{sec:learning},
which learns the geometry of the problem to some extent from the lengths 
of the projection steps carried out in the past and uses this knowledge 
to select favourable projections in the future.

\medskip

The notation used in this paper is mostly standard.
Given a point $x\in\R^d$ and a closed convex set $C\subset\R^d$,
it is well-known that the projection
\[\proj(x,C):=\argmin_{z\in C}\|x-z\|\]
of $x$ to $C$ exists and is a unique point.

By $\mathrm{randperm}(1,\ldots,N)$ we denote a permutation of the
numbers $1,\ldots,N$ which is sampled uniformly from the set
of all such permutations, and by $\mathrm{urs}(I)$, we denote a 
uniform random sample from an index set $I\subset\{1,\ldots,N\}$.

\subsection{The benchmark: MCP and MRP} \label{sec:benchmark}

The now classical method of cyclic projections, which was
originally published in \cite{Bregman}, approximates a point $x^*\in C$ 
by iteratively projecting to the sets $C_1,\ldots,C_N$ in a cyclic fashion,
see Algorithm \ref{alg:MCP}.

\begin{algorithm}
\caption{Method of cyclic projections (MCP)}
\label{alg:MCP}
\KwIn{$C_1,\ldots,C_N\subset\R^d$, $x_0\in\R^d$}
\For{$k\gets 0$ \KwTo $\infty$} {
	$x_{k+1}\gets\proj(x_k,C_{\mathrm{mod}(k,N)+1})$\;}	
\end{algorithm}

Algorithm \ref{alg:MCP} may converge very slowly when many of the 
projection steps are small.
This behavior may originate from an unfavorable ordering of the 
sets $C_1,\ldots,C_N$, which can be helped by randomly shuffling 
the order of the sets in every cycle, see Algorithm \ref{alg:MRP}.

The random Kaczmarz method proposed in \cite{Strohmer} is a prominent
variant of Algorithm \ref{alg:MRP} in the framework of row-action methods
for solving linear systems, which is known to converge in expectation.
Since MRP slightly outperformed the random Kaczmarz method in all examples
we have studied, we use MRP as the benchmark for randomized algorithms.

\begin{algorithm}
\caption{Method of randomized projections (MRP)}
\label{alg:MRP}
\KwIn{$C_1,\ldots,C_N\subset\R^d$, $x^{(0)}_0\in\R^d$}
\For{$k\gets 0$ \KwTo $\infty$} {
$\pi\gets\mathrm{randperm}(1,\ldots,N)$\;
\For{$j\gets 0$ \KwTo $N-1$}{
	$x^{(k)}_{j+1}\gets\proj(x^{(k)}_j,C_{\pi(j+1)})$\;}
$x^{(k+1)}_0\gets x^{(k)}_N$\;}
\end{algorithm}

\subsection{A projection algorithm with learning ability} \label{sec:learning}

The idea behind Algorithm \ref{alg:PAM} (PAM) is to keep a record of 
the lengths of projection steps performed in the past and to give preference 
to operations that have lead to large projection steps.
This enables our algorithm to learn to some extent the geometry of the
problem with manageable additional computational cost.

From a formalistic point of view, our approach resembles to some extent 
the techniques of \emph{loping} and \emph{flagging} introduced in
\cite{Elfving} in the setting of row-action methods.
These techniques suppress the effect of noise in the data on MCP by ignoring
projections which had very small residuals in previous cycles.
From a phenomenological perspective, however, these modifications
of MCP do not have much in common with PAM.

\medskip

\begin{algorithm}
\caption{Projection algorithm with memory (PAM)}
\label{alg:PAM}
\KwIn{$C_1,\ldots,C_N\subset\R^d$, 
$x_0\in C_1$, 
$D^0\in\R^{N\times N}_{\ge 0}$,
$\phi:\{1,\ldots,N\}\times\R^{N\times N}_{\ge 0}\to\R_{\ge 0}$}
$j_0\gets 1$\;
\For{$k\gets 0$ \KwTo $\infty$} {
  \tcc{carry out most promising admissible projection}
	$j_{k+1}\gets\mathrm{urs}(\argmax_{\ell\in\{1,\ldots,N\}
	\setminus\{j_k\}}D_{j_k,\ell}^k)$\;
	$x_{k+1}\gets\proj(x_k,C_{j_{k+1}})$\;
	\tcc{update distance matrix}	
	$D^{k+1}\gets D^k$\;
	$D_{j_k,j_{k+1}}^{k+1}\gets\max\{\|x_{k+1}-x_k\|,\phi(j_k,D^k)\}$\;}	
\end{algorithm}

In the following, we give an intuitive description how some of the
individual components interact in Algorithm \ref{alg:PAM}.
They will be treated with proper mathematical rigour in the next section.

\begin{itemize}
\item [i)] The sequence of matrices $(D^k)_{k\in\N}$ records -- up to
the impact of the function $\phi$ -- the length of the $k$-th
projection step from set $C_{j_k}$ to $C_{j_{k+1}}$ in the component 
$D^{k+1}_{j_k,j_{k+1}}$.
\item [ii)] The input $D^0$ has three distinct effects.
\begin{itemize}
\item [a)] If $D^0_{m,n}=0$, a transition from $C_m$ to $C_n$ will 
not occur during the entire runtime of the algorithm, 
see Lemma \ref{sparsity}.
Thus, by choosing a sparse $D^0$ as in Example \ref{some:matrices}(i), 
one can limit the amount of information that needs to be stored and
processed at runtime.
For a graphic illustration, see Example \ref{sparseD0},
and for the impact on performance in the context of a toy model, 
see Example \ref{bandwidth:err}.
\item [b)] The choice of $D^0$ can incorporate a priori knowledge:
The more likely a transition from $C_m$ to $C_n$ is to be beneficial,
the larger the entry $D^0_{m,n}$ should be chosen, see Example
\ref{some:matrices}(ii).
\item [c)] If the entries of $D^0$ are small relative to the first 
several lengths $\|x_{k+1}-x_k\|$ of steps to be carried out, the algorithm
will not perform well in an initial stage, see Example \ref{scaling}(ii).
If they are larger, the algorithm will initially behave like MRP, see
Example \ref{scaling}(i).
\end{itemize}
\item [iii)] The function $\phi$ modifies the step length
before it is recorded in the matrix $D^{k+1}$.
\begin{itemize}
\item [a)] It ensures that only strictly positive values are 
written into $D^{k+1}$.
\item [b)] It determines at what level of overall
performance a transition from $C_m$ to $C_n$ will get reactivated
after it generated a short step.
\end{itemize}
\end{itemize}

Finally, we would like to mention that we represent the recorded 
step-lengths in a matrix $D^k$ to keep the notation manageable.
Depending on the size of the problem, the sparsity pattern of $D^0$
and the policy $\phi$, it can be beneficial to use a different data
structure -- such as one red-black tree per set $C_j$ -- 
to store and search this data with moderate on-cost compared 
to MCP and MRP.

\section{Admissible input}

For Algorithm \ref{alg:PAM} to converge, we require the inputs
to have certain properties.
The matrix $D^0$ is required to be irreducible in the following sense.

\begin{definition}[admissible matrix]\label{def:irr}
A matrix $D\in\R^{N\times N}_{\ge 0}$ is called admissible if it
satisfies
\begin{itemize}
\item [i)] $D_{m,m}=0$ for all $m\in\{1,\ldots,N\}$, and
\item [ii)] for any indices $m,n\in\{1,\ldots,N\}$ with $m\neq n$, there
exist some $\ell\in\N$ and indices $i_1,\ldots,i_\ell\in\{1,\ldots,N\}$ 
such that
\[i_1=m,\quad i_\ell=n,\quad\text{and}\quad
D_{i_s,i_{s+1}}>0\quad\forall\,s\in\{1,\ldots,\ell-1\}.\]
\end{itemize}
\end{definition}

We give a few examples how the matrix $D^0$ can be chosen.

\begin{example}[some admissible matrices]\label{some:matrices}
For a good performance of Algorithm \ref{alg:PAM}, it is helpful to
multiply the matrices proposed below with a positive scalar to ensure
that their respective nonzero entries are -- at least on average and 
for small $k$ -- similar to or larger than the length $\|x_{k+1}-x_k\|$ 
of the $k$-th projection step from set $C_{j_k}$ to $C_{j_{k+1}}$.

\medskip

i) To limit the effective size of the matrices $D^k$, one can choose 
$D^0$ to be a sparse matrix such as the banded matrices
$D^\leftrightarrow\in\R^{N\times N}$ given by
\begin{equation*}
D^\leftrightarrow_{m,n}=\begin{cases}
1,& 0<|n-m|\le\omega\ \text{or}\ N+m-n\le\omega\ \text{or}\
N+n-m\le\omega,\\0,& \text{otherwise}
\end{cases}\end{equation*}
and $D^\rightarrow\in\R^{N\times N}$ given by
\begin{equation*}
D^\rightarrow_{m,n}=\begin{cases}
1,& 0<n-m\le\omega\ \text{or}\ N+n-m\le\omega,\\
0,& \text{otherwise}
\end{cases}\end{equation*}
with some $\omega\in\N$ with $1\le\omega\ll N$.

\medskip

ii) In scenarios, where the concept of an angle makes sense, it is
reasonable to work with a matrix $D^\angle\in\R^{N\times N}$ given by
\[D_{m,n}^\angle=\begin{cases}\gg 1,&\angle(C_m,C_n)\ 
\text{known to be large},\\
1,&\angle(C_m,C_n)\ \text{unknown},\\
\ll 1,&\angle(C_m,C_n)\ \text{known to be small},\\
0,& m=n\end{cases}\]
to introduce a bias in favour of transitions with large angles, 
which are more likely to result in large step-lengths.

\medskip

It is easy to check that the above matrices are admissible.
Please note that MCP is a special case of Algorithm \ref{alg:PAM},
which can be realized by choosing $D^0$ to be the matrix 
$D^{\rightarrow}$ with $\omega=1$.
\end{example}

The function $\phi$ is required to be strictly positive on all 
meaningful input, and it must ensure a certain decay of the 
entries of $D$.

\begin{definition}[admissible policies]\label{policies}
A function 
\[\phi:\{1,\ldots,N\}\times\R^{N\times N}_{\ge 0}\to\R_{\ge 0}\]
is called an admissible policy if there exists $\beta\in(0,1)$ such that
\begin{itemize}
\item [i)] $\phi(m,D)>0$ holds for all $m\in\{1,\ldots,N\}$ and
$D\in\R^{N\times N}_{\ge 0}$ satisfying $\max_nD_{m,n}\neq 0$, and
\item [ii)] $\phi(m,D)\le\beta\max_nD_{m,n}$
for all $m\in\{1,\ldots,N\}$ and $D\in\R^{N\times N}_{\ge 0}$.
\end{itemize}
\end{definition}

We propose some particular policies $\phi$.

\begin{example}[some admissible policies]
It is easy to check that both policies proposed below are indeed admissible
for every $\beta\in(0,1)$.

\medskip

i) The function
\[\phi_{\min}(m,D):=\beta\min_{\{n:D_{m,n}>0\}}D_{m,n}\]
ensures that the number which is written into the distance
matrix $D^{k+1}$ in line 7 of Algorithm \ref{alg:PAM} is at least $\beta$
times the minimal previously recorded step-length from $C_m$
to another $C_n$.

\medskip

ii) The function
\[\phi_{\text{av}}(m,D):=\frac{\beta}{\#\{n:D_{m,n}>0\}}\sum_{\{n:D_{m,n}>0\}}D_{m,n}\]
ensures that the number which is written into $D^{k+1}$ in line 7 of Algorithm \ref{alg:PAM} is at least $\beta$ times the average of the
previously recorded step-lengths from $C_m$ to another $C_n$.

\medskip

Note that the value of the minimal nonzero entry in a
row as well as the average of the nonzero entries in a row
can be updated with negligible computational cost in every step.
\end{example}

The proof of the following statement is elementary.

\begin{lemma}[preservation of sparsity pattern]\label{sparsity}
Let both $D^0\in\R^{N\times N}_{\ge 0}$ and 
$\phi:\{1,\ldots,N\}\times\R^{N\times N}_{\ge 0}\to\R_{\ge 0}$ 
be admissible, and let $(D^k)_{k\in\N}\in(\R^{N\times N}_{\ge 0})^\N$ 
be the matrices generated by Algorithm \ref{alg:PAM} with arbitrary initial 
value $x\in\R^d$.
Then for any $k\in\N$ and $m,n\in\{1,\ldots,N\}$, we have
\[D^k_{m,n}>0\quad\text{if and only if}\quad D^0_{m,n}>0,\]
and, in particular, the matrices $D^k$ are admissible for all $k\in\N$.
\end{lemma}

\section{Convergence analysis}

We first prove a general principle for projection algorithms
in Section \ref{sec:principle}.
Then we show in Section \ref{sec:satisfies} that Algorithm
\ref{alg:PAM} satisfies the assumptions of this statement.

\subsection{Recurrence implies convergence}\label{sec:principle}

We restate a slightly modified version of Corollaries 1 and 2 from \cite{Bregman}.

\begin{lemma}[projections reduce error] \label{Bregman:lemma}
Let $C_1,\ldots,C_N\subset\R^d$ be closed convex sets and
$z\in\cap_{j=1}^NC_j$, and let the sequences
$(j_k)_{k\in\N}\in\{1,\ldots,N\}^\N$ and $(x_k)_{k\in\N}\in(\R^d)^\N$ satisfy
\[x_{k+1}=\proj(x_k,C_{j_k})\quad\forall\,k\in\N.\]
Then we have
\begin{align}
&\|x_{k+1}-z\|^2\le\|x_k-z\|^2-\|x_{k+1}-x_k\|^2\quad\forall\,k\in\N,\label{ineq}\\
&\|x_{k+1}-z\|\le\|x_k-z\|\le\|x_0-z\|\quad\forall\,k\in\N,\label{bounded}
\end{align}
\end{lemma}


Now we show that every projection algorithm, which projects to
every set $C_j$ infinitely often, generates a sequence that converges
to a point in $C$.
Related results are known in the community working on firmly nonexpansive
operators, see Theorem 4.1 from \cite{Dao}.
We include an explicit statement of this fact and an elementary proof 
to keep the paper self-contained. 

\begin{proposition}[recurrence implies convergence]
\label{recurrence:convergence}
Let $C_1,\ldots,C_N\subset\R^d$ be closed convex sets with 
$\cap_{j=1}^NC_j\neq\emptyset$, and let $(j_k)_{k\in\N}\in\{1,\ldots,N\}^\N$ 
and $(x_k)_{k\in\N}\in(\R^d)^\N$ be sequences which satisfy
\[x_{k+1}=\proj(x_k,C_{j_k})\quad\forall\,k\in\N\]
as well as the recurrence condition
\begin{equation}\label{many:of:each}
\#\{k\in\N:j_k=j\}=\infty\quad\forall\,j\in\{1,\ldots,\N\}.
\end{equation}
Then there exists $x^*\in\cap_{j=1}^NC_j$ such that $\lim_{k\to\infty}x_k=x^*$.
\end{proposition}

\begin{proof}
Because of statement \eqref{bounded} of Lemma \ref{Bregman:lemma}, 
there exist a subsequence $(x_{k_\ell})_{\ell\in\N}$ of $(x_k)_{k\in\N}$ 
and $x^*\in\R^d$ such that
\begin{equation}\label{convergent:ss}
\lim_{\ell\to\infty}\|x_{k_\ell}-x^*\|=0.
\end{equation}
Clearly, there exist $j^*\in\{1,\ldots,N\}$ and a subsequence $(k_{\ell_m})_{m\in\N}$ of the sequence $(k_\ell)_{\ell\in\N}$ with 
\[j_{k_{\ell_m}}=j^*\quad\forall\,m\in\N.\]
Since $C_{j^*}$ is closed, we have $x^*\in C_{j^*}$.
We partition $\{1,\ldots,N\}$ into 
\[J^*:=\big\{j\in\{1,\ldots,N\}: x^*\in C_j\big\},\quad 
J_*=\{1,\ldots,N\}\setminus J^*.\]
By the above, we have $J^*\neq\emptyset$.
Assume that $J_*\neq\emptyset$. 
By induction, using statement \eqref{many:of:each}, we can construct 
sequences $(k'_\ell)_{\ell\in\N}\in\N^\N$ and 
$(k''_\ell)_{\ell\in\N}\in\N^\N$ given by $k'_0:=k_{\ell_0}$ and
the iteration
\begin{align*}
&k''_{\ell}:=\min\{k\in\N:\ k>k'_{\ell},\ j_k\in J_*\},\\
&k'_{\ell+1}:=\min\{k_{\ell_m}:m\in\N,\ k_{\ell_m}>k''_\ell\}
\end{align*}
for $\ell\in\N$.
In particular, we have
\[k_0'<k_0''<k_1'<k_1''<\ldots\]
Since $J_*$ is finite, there exists $\eps>0$ such that 
\[\dist(x^*,C_j)\ge2\eps\quad\forall\,j\in J_*.\]
By construction of the sequence $(k_\ell')_{\ell\in\N}$, there exists 
$\ell^*\in\N$ such that
\[\|x_{k_\ell'}-x^*\|\le\eps\quad\forall\,\ell\ge\ell^*.\]
Applying statement \eqref{bounded} of Lemma \ref{Bregman:lemma} 
with $z=x^*$ and the system of sets $\{C_j:j\in J^*\}$, and using the construction 
of the sequence $(k_\ell'')_{\ell\in\N}$, we obtain
\[\|x_k-x^*\|\le\|x_{k_\ell'}-x^*\|\le\eps\quad
\forall\,\ell\ge\ell^*,\ \forall\,k\in[k_\ell',k_\ell'').\]
On the other hand, we have $\|x_{k_\ell''}-x^*\|\ge 2\eps$
for all $\ell\in\N$, so
\begin{equation}\label{small:step}
\|x_{k_\ell''}-x_{k_\ell''-1}\|
\ge\|x_{k_\ell''}-x^*\|-\|x^*-x_{k_\ell''-1}\|
\ge\eps\quad\forall\,\ell\ge\ell^*.
\end{equation}
Now let $z\in\cap_{j=1}^NC_j$ and use statement \eqref{small:step} and
statement \eqref{ineq} from Lemma \ref{Bregman:lemma} multiple times to obtain
\[\lim_{k\to\infty}\|x_k-z\|^2
\le\|x_0-z\|^2-\lim_{k\to\infty}\sum_{j=0}^{k-1}\|x_{j+1}-x_j\|^2
=-\infty,\]
which is a contradiction.
Hence $J_*=\emptyset$, and $x^*\in\cap_{j=1}^NC_j$.
Now statements \eqref{convergent:ss} and statement \eqref{bounded} 
of Lemma \ref{Bregman:lemma} with $z=x^*$ imply 
$\lim_{k\to\infty}x_k=x^*$, as desired.
\end{proof}

\subsection{Convergence of PAM} \label{sec:satisfies}

We check that Algorithm \ref{alg:PAM} satisfies the assumptions
of Proposition \ref{recurrence:convergence}, whenever the matrix $D^0$
and the policy $\phi$ are admissible.

\begin{proposition}[PAM is recurrent] \label{all:numbers:to:zero}
Let $C_1,\ldots,C_N\subset\R^d$ be closed convex sets which satisfy
$\cap_{j=1}^NC_j\neq\emptyset$, and let the matrix 
$D^0\in\R^{N\times N}_{\ge 0}$ and the policy 
$\phi:\{1,\ldots,N\}\times\R^{N\times N}_{\ge 0}\to\R_{\ge 0}$
be admissible.
Then for any initial point $x_0\in C_1$, the sequences 
$(j_k)_{k\in\N}\in\{1,\ldots,N\}^\N$ and 
$(D^k)_{k\in\N}\in(\R_{\ge 0}^{N\times N})^\N$  
generated by Algorithm \ref{alg:PAM} satisfy
\begin{align}
&\lim_{k\to\infty}\max_{m,n\in\{1,\ldots,N\}}D_{m,n}^k=0,\label{DtoZ}\\
&\#\{k\in\N:j_k=m\}=\infty\quad\forall\,m\in\{1,\ldots,\N\}.\label{rec}
\end{align}
\end{proposition}

\begin{proof}
Let $z\in\cap_{j=1}^NC_j$.
Applying inequality \eqref{ineq} from Lemma \ref{Bregman:lemma} 
multiple times yields
\[0\le\|x_k-z\|^2\le\|x_0-z\|^2-\sum_{j=0}^{k-1}\|x_{j+1}-x_j\|^2\quad\forall\,k\in\N,\]
which forces
\begin{equation}\label{to:zero}
\lim_{k\to\infty}\|x_{k+1}-x_k\|=0.
\end{equation}
Let us denote 
\[J_\infty:=\{m\in\{1,\ldots,N\}:\#\{k\in\N:j_k=m\}=\infty\}.\]
Obviously, we have $J_\infty\neq\emptyset$.
Let $m\in J_\infty$, and let $(k_\ell)_{\ell\in\N}\in\N^{\N}$ be the
maximal strictly increasing sequence with $j_{k_\ell}=m$ for all 
$\ell\in\N$. 
Let $\eps>0$. 
By statement \eqref{to:zero}, there exists $k^*\in\N$ such that 
\[\|x_{k+1}-x_k\|\le\eps\quad\text{for all}\quad k\ge k^*.\]
Because of line 6 of Algorithm \ref{alg:PAM} and since $\phi$ is admissible
with a decay rate $\beta\in(0,1)$, we have
\begin{equation}\label{decline}
\max_nD^{k'}_{m,n}\le\max\{\eps,\max_nD^k_{m,n}\}\quad\text{whenever}\quad
k^*\le k\le k'.
\end{equation}
Now let $\ell\in\N$ be such that $k_\ell\ge k^*$.
We wish to show that
\begin{equation}\label{N:down}
\max_nD^{k_{\ell+N}}_{m,n}
\le\max\{\eps,\beta\max_nD^{k_\ell}_{m,n}\}.
\end{equation}
To this end, we introduce the quantity
\[\nu(p):=\#\{n:D^{k_{\ell+p}}_{m,n}
>\max\{\eps,\beta\max_{n'}D^{k_\ell}_{m,n'}\}\]
and prove the statement
\begin{equation}\label{istat}
\nu(p)\le N-p\quad\text{for}\quad p\in\{0,\ldots,N\}
\end{equation}
by induction.
Statement \eqref{istat} is trivial for $p=0$.
Assume that statement \eqref{istat} holds for some $p\in\{0,\ldots,N-1\}$.
If $\nu(p)=0$, then statement \eqref{decline} implies that $\nu(p+1)=0$,
and that the induction hypothesis \eqref{istat} holds for $p+1$.
If $\nu(p)>0$, then line 3 of Algorithm \ref{alg:PAM} selects an index
\[j_{k_{\ell+p}+1}
\in\argmax_{q\in\{1,\ldots,N\}\setminus\{m\}}D^{k_{\ell+p}}_{m,q}\]
that satisfies
\[D^{k_{\ell+p}}_{m,j_{k_{\ell+p}+1}}>\max\{\eps,\beta\max_{n'}D^{k_\ell}_{m,n'}\}.\]
By line 6 of Algorithm \ref{alg:PAM} and by statemenr \eqref{decline}, 
we have
\begin{align*}
D^{k_{\ell+p}+1}_{m,j_{k_{\ell+p}+1}}
&=\max\{\|x_{k_{\ell+p}+1}-x_{k_{\ell+p}}\|,\phi(m,D^{k_{\ell+p}})\}\\
&\le\max\{\eps,\beta\max_nD^{k_{\ell+p}}_{m,n}\}
\le\max\{\eps,\beta\max_nD^{k_\ell}_{m,n}\}.
\end{align*}
By construction of the sequence $(k_\ell)_{\ell\in\N}$,
it follows that
\[D^{k_{\ell+p+1}}_{m,n}=D^{k_{\ell+p}+1}_{m,n}\quad\forall\,n\in\{1,\ldots,N\},\]
so $\nu(p+1)=\nu(p)-1$, and statement \eqref{istat} holds for $p+1$.
This completes the induction, and statement \eqref{N:down} is verified,
because $\nu(N)=0$.
Since $\beta<1$, statements \eqref{decline} and \eqref{N:down} imply that
there exists $k^{**}\in\N$ such that $\max_nD^k_{m,n}\le\eps$ for all
$k\ge k^{**}$.
Since $m\in J_\infty$ and $\eps>0$ were arbitrary, we have shown that
\begin{equation}\label{JtoZ}
\lim_{k\to\infty}\max_nD^k_{m,n}\to 0\quad\forall\,m\in J_\infty.
\end{equation}
In view of Lemma \ref{sparsity}, this implies
\begin{equation}\label{ninJ}
n\in J_\infty\quad \text{whenever}\quad m\in J_\infty\ \text{and}\ 
D^0_{m,n}>0.
\end{equation}
Since $D^0$ satisfies part ii) of Definition \ref{def:irr}, a simple
recursion on statement \eqref{ninJ} yields $J_\infty=\{1,\ldots,N\}$,
which is statement \eqref{rec}.
Consequently, statement \eqref{JtoZ} implies \eqref{DtoZ}.
\end{proof}

Now we summarize the above in the main theoretical result of this paper.

\begin{theorem}[convergence of PAM]\label{convergence:theorem}
Let $C_1,\ldots,C_N\subset\R^d$ be closed convex sets which satisfy
$\cap_{j=1}^NC_j\neq\emptyset$, and let the matrix 
$D^0\in\R^{N\times N}_{\ge 0}$ and the policy 
$\phi:\{1,\ldots,N\}\times\R^{N\times N}_{\ge 0}\to\R_{\ge 0}$
be admissible.
Then there exists $x^*\in\cap_{j=1}^NC_j$
such that the sequence $(x_k)_{k\in\N}\in(\R^d)^\N$ generated by Algorithm
\ref{alg:PAM} satisfies 
\[\lim_{k\to\infty}x_k=x^*.\]
\end{theorem}

\begin{proof}
Proposition \ref{all:numbers:to:zero} verifies that Algorithm \ref{alg:PAM}
satisfies the assumptions of Proposition \ref{recurrence:convergence}, 
so PAM is indeed convergent.
\end{proof}

\section{An instructive toy example}

We explore the performance of PAM with different matrices $D^0$
and policies $\phi$ in a very simple toy example, and compare 
its behavior with MCP and MRP.
We are fully aware that this example has many unrealistic features,
but it allows us to illustrate key features of our algorithm 
in a nice graphic way.
To keep things simple, we measure the computational cost of all
three algorithms in the number of iterations, which is the number
of projection steps carried out.

\medskip

Throughout this section, we consider the one-dimensional subspaces
\[C_j:=\{s\begin{pmatrix}r\cos(\tfrac{j\pi}{N})\\
r\sin(\tfrac{j\pi}{N})\\1\end{pmatrix}:s\in\R\},\quad
j=1,\ldots,N,\]
with $\cap_{j=1}^NC_j=\{0\}$, and the initial point 
$x_0=(\cos(\frac{\pi}{N}),\sin(\frac{\pi}{N}),1)$.
For aesthetical reasons, we choose $N=9$ and $r=0.05$ in most illustrations.

\begin{figure}[ht] 
\centering\includegraphics{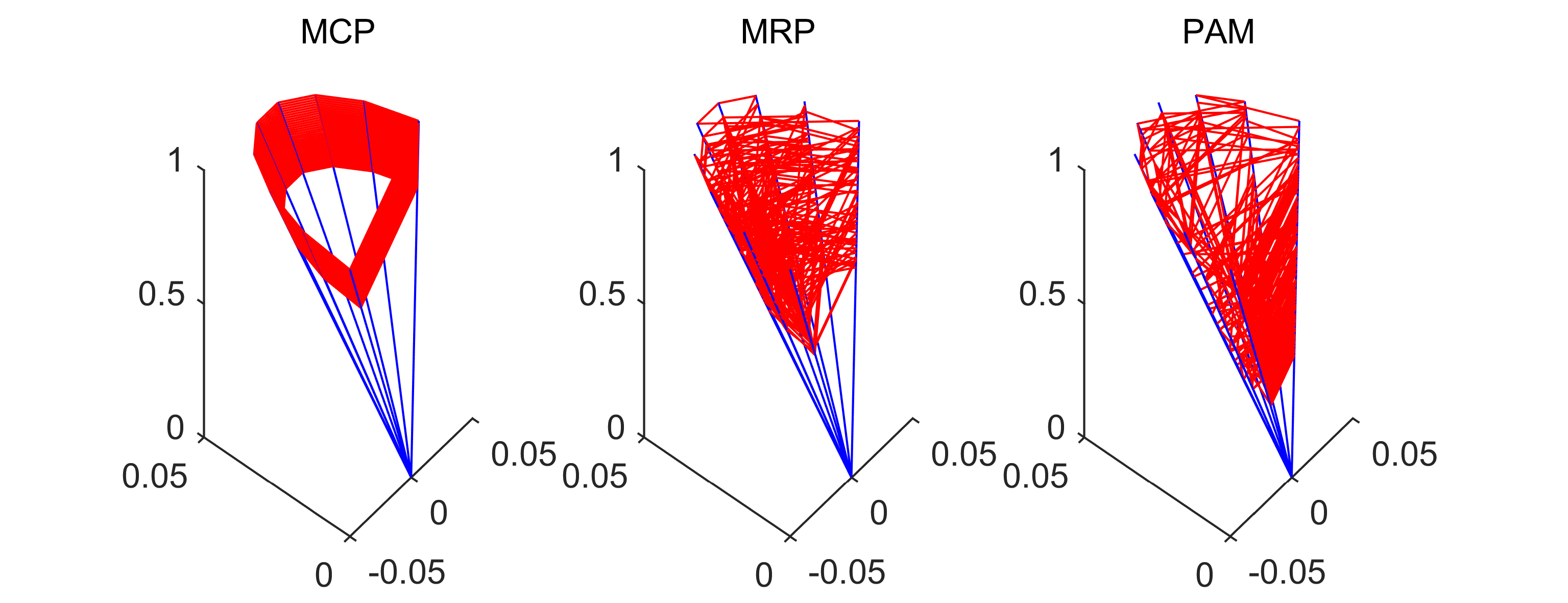}

\vspace{1ex}

\centering\includegraphics{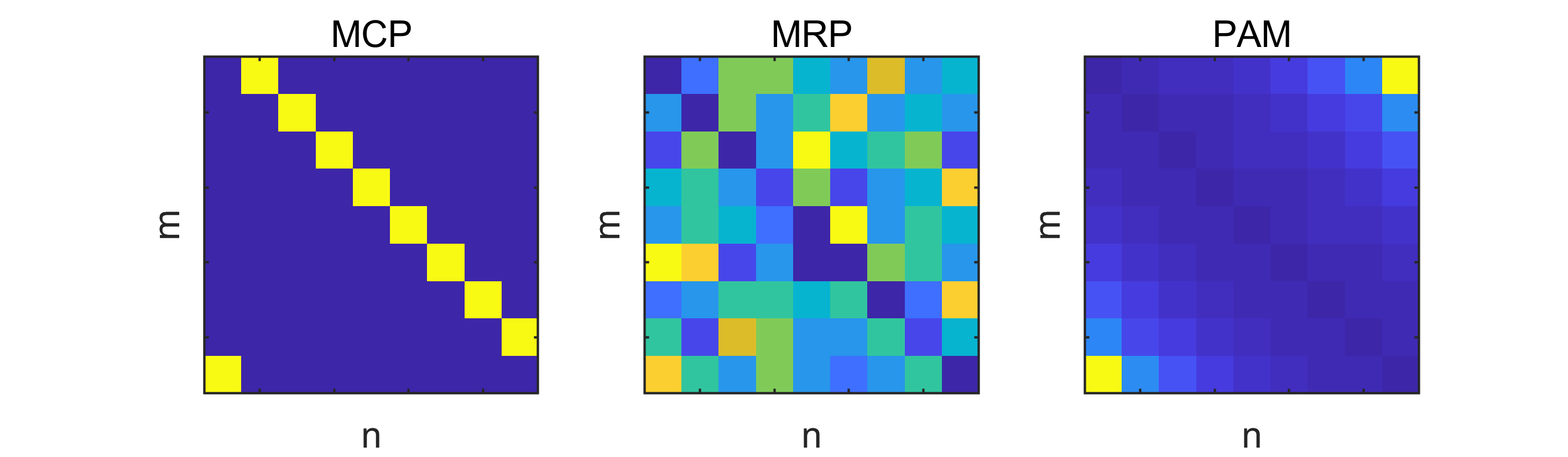}

\vspace{-1.5ex}

\caption{Methods MCP, MRP and PAM applied to toy problem.
Top row: Iterates red, subspaces blue. 
Bottom row: Frequencies (yellow=high, blue=low) of transitions from 
set $C_m$ to set $C_n$.
\label{benchmark}}
\end{figure} 

\medskip

Let us first compare MCP, MRP and PAM without going into too much technical
detail.

\begin{example}[benchmark versus PAM]\label{ex:benchmark}
In Figure \ref{benchmark},
we see at a glance how the strategies behind MCP, MRP and PAM impact
their behavior and performance when applied to the toy model.
The fixed order of projections in MCP can result in significant
underperformance, while the random order of the projections in MRP
guarantees that the average of the achievable progress is realized.

This motivates us to try and outperform the average by assigning
a high probability to transitions, which performed better than
average in previous iterations, in the new method PAM.
The toy problem suggests that this is not a bad idea, when 
the matrix $D^0$ and the policy $\phi$ are chosen well for the problem 
at hand.
For this showcase, we used $D^0$ and $\phi$ as in Example \ref{scaling}a),
and carried out 315 iterations with each method.
\end{example}

\begin{figure}[p] 
\centering\includegraphics{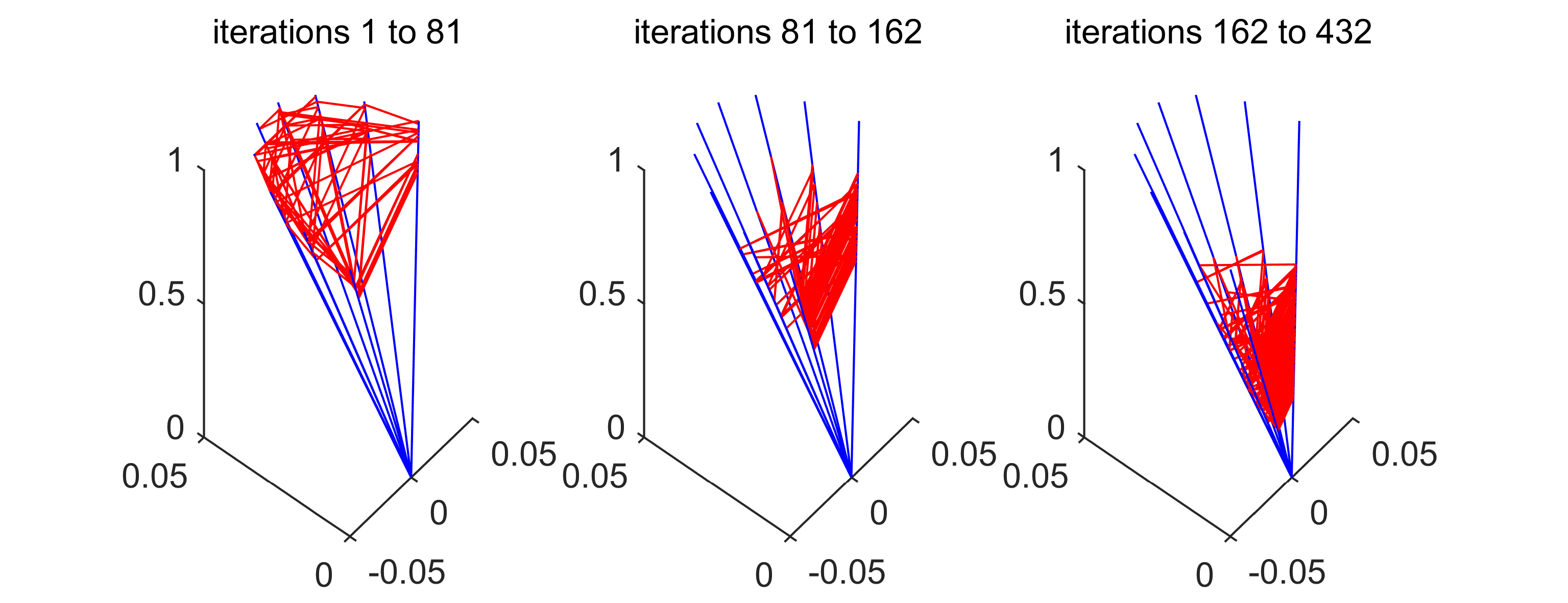}

\centering\includegraphics{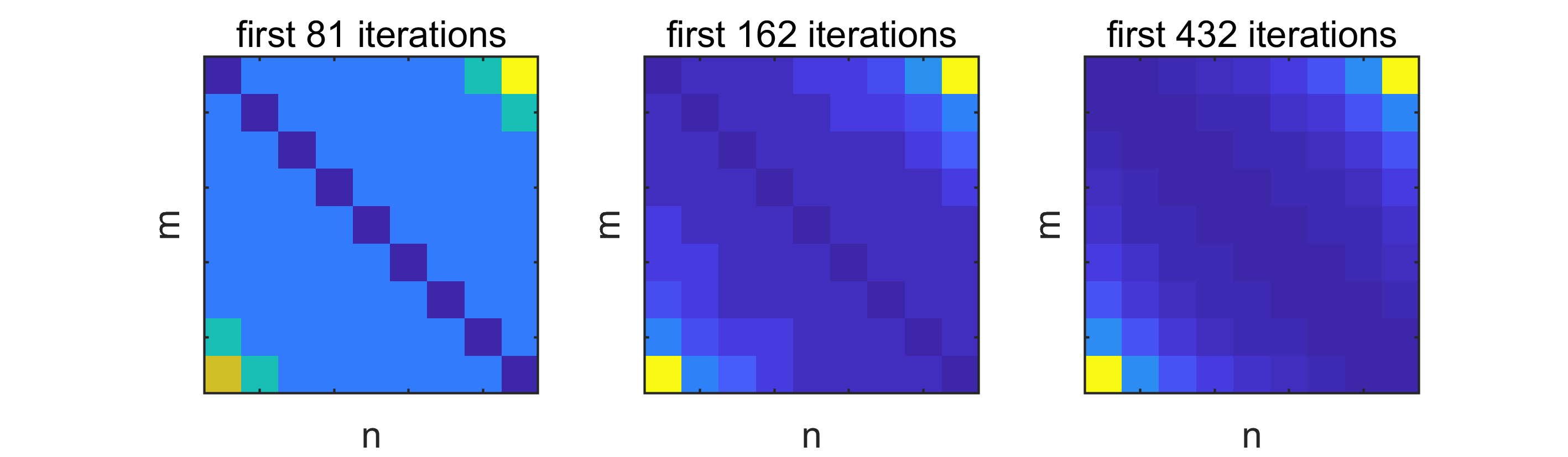}

\vspace{-1.5ex}

\caption{Trajectories and frequencies of PAM as in Example \ref{scaling}(i).
\label{Inf}}
\end{figure} 

\begin{figure}[p] 
\centering\includegraphics{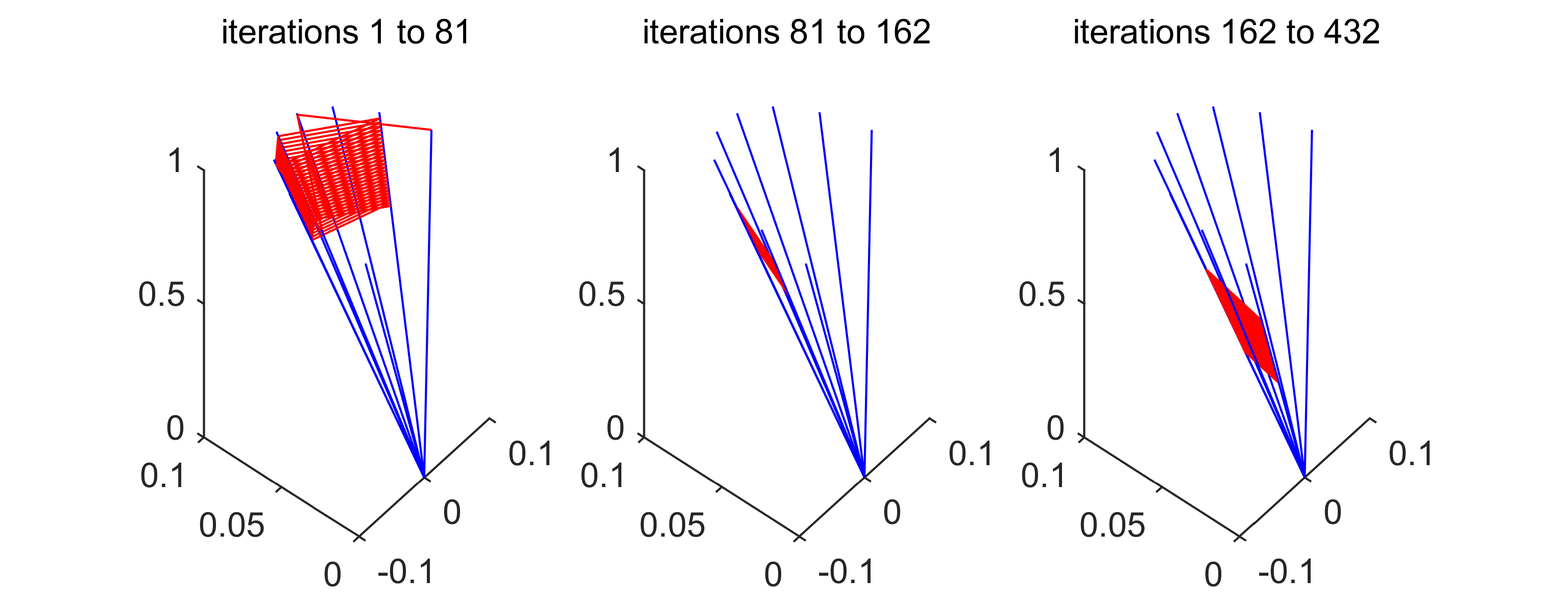}

\centering\includegraphics{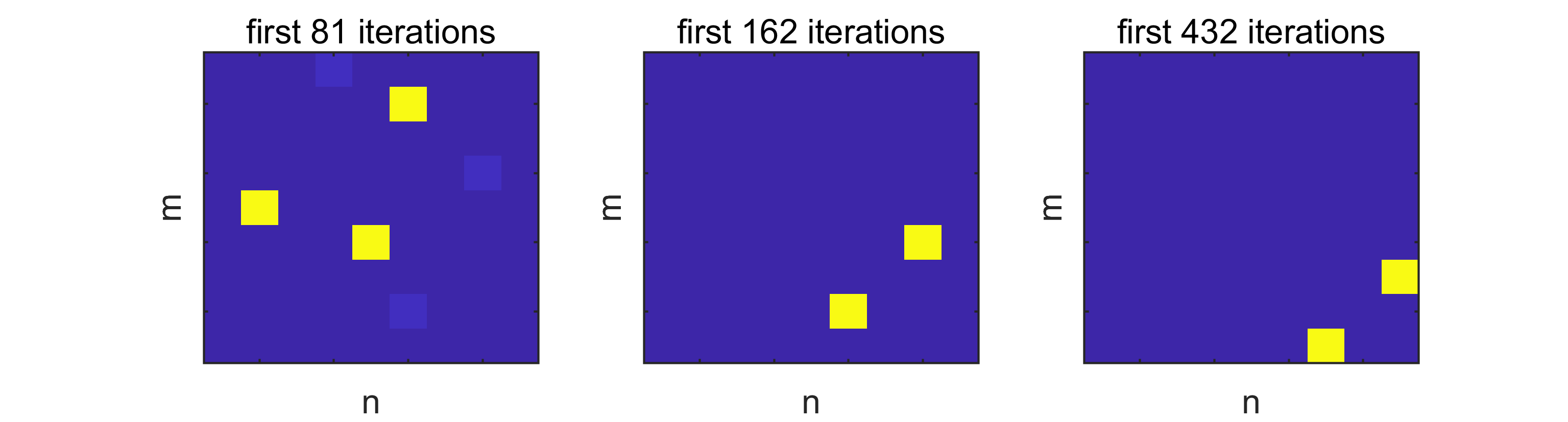}

\vspace{-1.5ex}

\caption{Trajectories and frequencies of PAM as in Example \ref{scaling}(ii).
\label{unf}}
\end{figure} 

In Example \ref{scaling}, we examine how a good performance of PAM 
can be achieved by a proper scaling of the initial matrix $D^0$.
Please note that the intention of this example is not to discuss
the size of numerical errors, but rather the qualitative behavior 
of PAM.
The matrices and the iteration numbers are chosen in such a way that 
these characteristics become clearly visible.

\begin{example}[scaling $D^0$]\label{scaling}
i) In Figure \ref{Inf},
we apply PAM with $\phi_{\min}$, $\beta=0.01$ and initial matrix 
$D^0\in\R^{N\times N}_{\ge 0}$ given by
\[D^0_{m,n}=\begin{cases}1,&m\neq n,\\0,&\text{else.}\end{cases}\]
While the entries of the matrix $D^k$ are large compared to the actual
step-sizes of the algorithm, we see a more or less uniform sampling
of the transitions, similar to the behavior of the superior benchmark
method MRP.
Once the sizes of the entries of the matrix $D^k$ are similar to the 
sizes of the steps carried out, PAM has learned the geometry of the problem
and focusses with high probability on profitable transitions, which
allows it to outperform MRP.

\medskip

ii) In Figure \ref{unf},
we apply PAM with $\phi_{\min}$, $\beta=0.01$ and initial matrix 
$D^0\in\R^{N\times N}_{\ge 0}$ given by
\[D^0_{m,n}=\begin{cases}0.01,&m\neq n,\\0,&\text{else,}\end{cases}\]
so the entries of the matrices $D^k$ underestimate the actual
step-sizes in the initial phase of the algorithm. 
This leads to unpredictable qualitative behaviour of PAM and incomplete
exploration of the admissible transitions, and in many 
cases to an underperformance relative to MRP.
Once the sizes of the entries of the matrix $D^k$ are similar to the 
sizes of the steps carried out, the qualitative behavior will be 
as in part i) above.
\end{example}

\begin{figure}[t] 
\centering\includegraphics{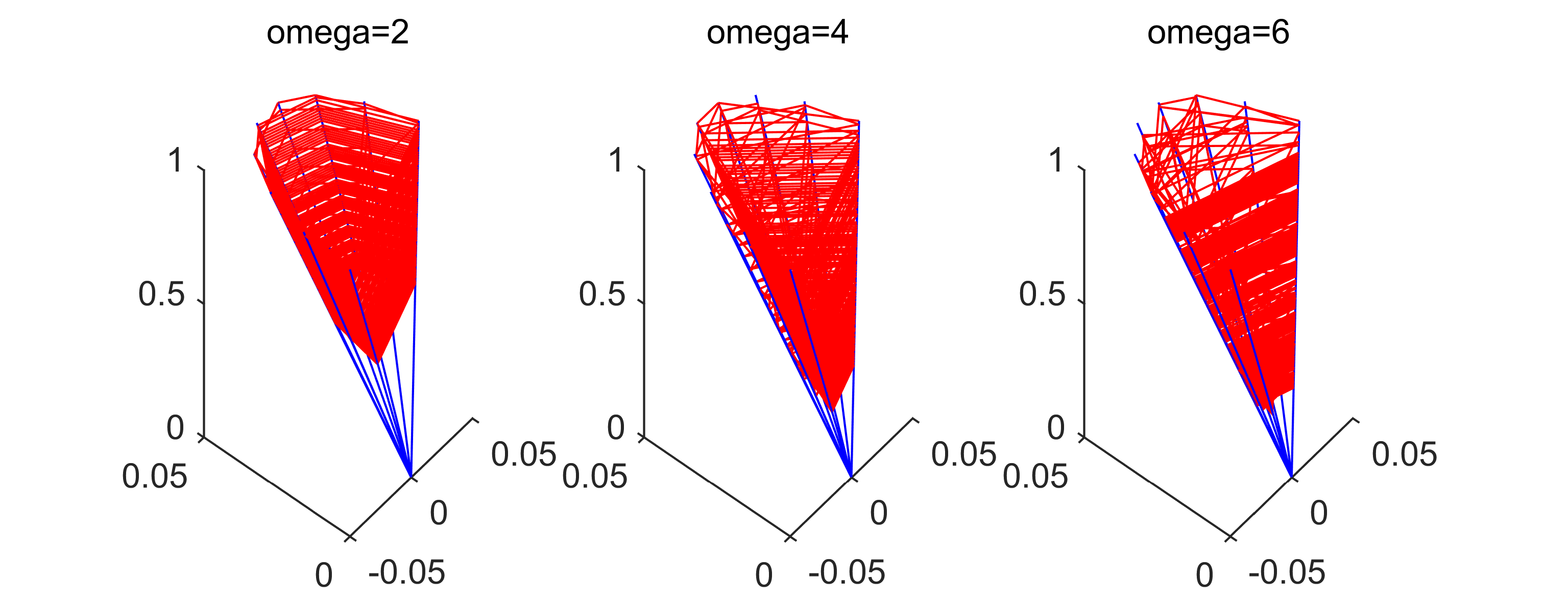}

\centering\includegraphics{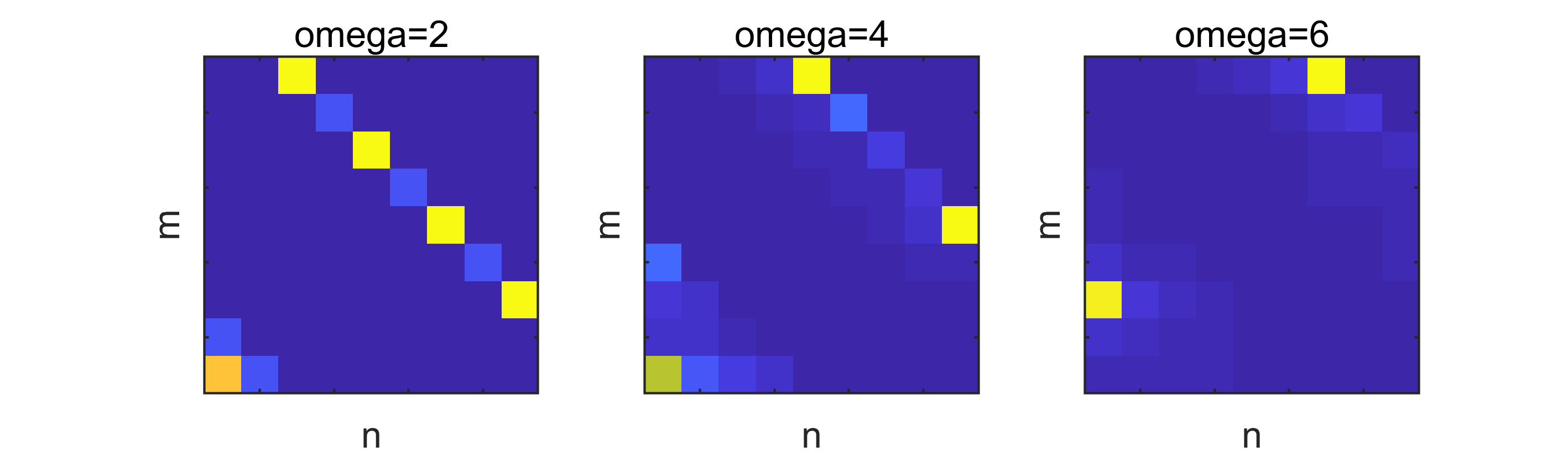}

\vspace{-1.5ex}

\caption{Trajectories and frequencies of PAM as in Example \ref{sparseD0}.
\label{omega}}
\end{figure} 

The effect of choosing a sparse $D^0$ is not surprising. 

\begin{example}[sparse $D^0$]\label{sparseD0}
We apply PAM to the model problem with $\phi_{\min}$, $\beta=0.01$ and the matrix $D^\rightarrow$ from Example \ref{some:matrices} with parameters
$\omega=2,4,6$, and obtain the results shown in Figure \ref{omega} 
after 432 iterations.
Note that the matrix $D^\rightarrow$ overestimates the first step-lengths
of the algorithm and therefore needs no scaling. 

The algorithm behaves exactly as expected:
After an initial learning phase, PAM focusses on the most profitable
admissible transitions.
A small bandwidth $\omega$ results in a shorter initial learning
phase, but small gain in long-term performance as compared to MCP.
On the other hand, a large $\omega$ results in a longer learning phase
with a seizable long-term gain in performance.
\end{example}

Our toy model is not sophisticated enough to reveal a significant
difference between the behavior induced by different policies $\phi$.
We can, however, observe how the choice of the bandwidth of the matrix 
$D^\rightarrow$ from Example \ref{some:matrices} impacts the performance
of the method in this particular example.

\begin{figure}[t] 
\centering\includegraphics{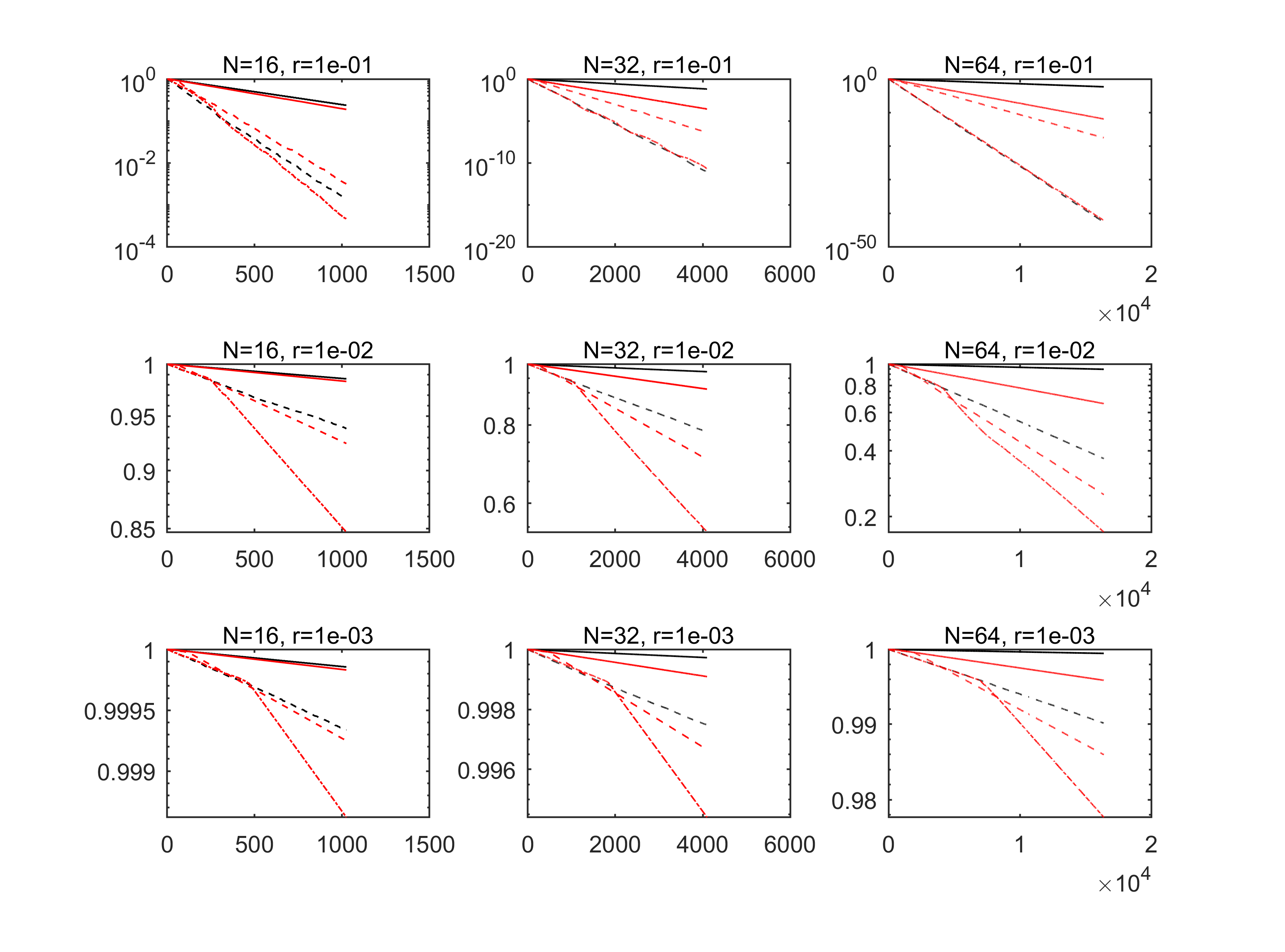}

\vspace{-1.5ex}

\caption{Error plots of methods applied to toy model with varying
parameters. 
Solid black line MCP, dashed black line MRP,
solid red line PAM $\omega=N/4$, dashed red line PAM $\omega=N/2$,
dash-dotted red line PAM $\omega=N$.
More details given in Example \ref{bandwidth:err}.
\label{errorplots}}
\end{figure} 

\begin{example}[first quantitative tests in toy example]
\label{bandwidth:err}
We apply MCP, MRP and PAM with initial matrix $D^\rightarrow$ from 
Example \ref{some:matrices} and three different choices of the 
bandwidth $\omega$ to our toy problem.
There are a few interesting features of the results displayed in 
Figure \ref{errorplots} we wish to summarize:
\begin{itemize}
\item [i)] The initial learning phase in which PAM explores the geometry 
of the problem is clearly visible in the error plot.
\item [ii)] When $\omega=N$, i.e.\ when every transition from set $C_m$
to set $C_n$ with $m\neq n$ is admissible, PAM never performed worse 
than MRP.
\item [iii)] The harder the problem is to solve for MCP and MRP 
(in this example this is the case when $r>0$ is small), 
the more clearly PAM (with large $\omega$) outperforms both methods.
\end{itemize}
\end{example}

\section{Conclusion}

This paper introduces the idea of learning to the realm of
algorithms for feasibility problems.
The focus is on establishing a first feasible algorithm 
and proving its convergence for a range of admissible 
learning strategies.
Since it was a major effort and achievement to quantify the speed of
convergence for MCP and MRP in the setting of affine subspaces, it seems
impossible to achieve something similar for PAM, which is, in a sense, 
path-dependent.
For this reason, we believe that we completed the theoretical
analysis of PAM in the present paper.

First experiments with PAM applied to computerized and 
seismic tomography data reveal that the performance of PAM varies
between different types of problems.
The choice of the matrix $D^0$ and the strategy $\phi$ really seems 
to matter in a real-world context, which calls for a detailed
computational investigation of the performance of PAM.
As the numerical handling of these problems is a challenge in itself,
and unrelated to the key issue of the present paper, we postpone a detailed
exploration of this issue to future work.

\subsection*{Acknowledgement}

The author thanks Matthew Tam for an introduction to the world 
of projection methods and support during the preparation of this paper.

\bibliographystyle{plain}
\bibliography{accelerated}

\begin{thebibliography}{10}

\bibitem{Li}
F.~Arroyo, E.~Arroyo, X.~Li, and J.~Zhu.
\newblock The convergence of the block cyclic projection with an overrelaxation
  parameter for compressed sensing based tomography.
\newblock {\em J. Comput. Appl. Math.}, 280:59--67, 2015.

\bibitem{Bai}
Z.~Bai and W.~Wu.
\newblock On greedy randomized {K}aczmarz method for solving large sparse
  linear systems.
\newblock {\em SIAM J. Sci. Comput.}, 40(1):A592--A606, 2018.

\bibitem{BB}
H.H. Bauschke and J.M. Borwein.
\newblock On projection algorithms for solving convex feasibility problems.
\newblock {\em SIAM Rev.}, 38(3):367--426, 1996.

\bibitem{Bauschke:Combettes}
H.H. Bauschke and P.L. Combettes.
\newblock {\em Convex analysis and monotone operator theory in {H}ilbert
  spaces}.
\newblock CMS Books in Mathematics/Ouvrages de Math\'{e}matiques de la SMC.
  Springer, New York, 2011.

\bibitem{Bauschke}
H.H. Bauschke, F.~Deutsch, H.~Hundal, and S.-H. Park.
\newblock Accelerating the convergence of the method of alternating
  projections.
\newblock {\em Trans. Amer. Math. Soc.}, 355(9):3433--3461, 2003.

\bibitem{Guoyin}
J.M. Borwein, G.~Li, and L.~Yao.
\newblock Analysis of the convergence rate for the cyclic projection algorithm
  applied to basic semialgebraic convex sets.
\newblock {\em SIAM J. Optim.}, 24(1):498--527, 2014.

\bibitem{Bregman}
L.M. Br\`egman.
\newblock Finding the common point of convex sets by the method of successive
  projection.
\newblock {\em Dokl. Akad. Nauk SSSR}, 162:487--490, 1965.

\bibitem{Censor}
Y.~Censor.
\newblock Row-action methods for huge and sparse systems and their
  applications.
\newblock {\em SIAM Rev.}, 23(4):444--466, 1981.

\bibitem{Dao}
M.N. Dao and M.K. Tam.
\newblock Union {A}veraged {O}perators with {A}pplications to {P}roximal
  {A}lgorithms for {M}in-{C}onvex {F}unctions.
\newblock {\em J. Optim. Theory Appl.}, 181(1):61--94, 2019.

\bibitem{Deutsch}
F.~Deutsch.
\newblock {\em Best approximation in inner product spaces}, volume~7 of {\em
  CMS Books in Mathematics/Ouvrages de Math\'{e}matiques de la SMC}.
\newblock Springer-Verlag, New York, 2001.

\bibitem{Elfving}
T.~Elfving, P.C. Hansen, and T.~Nikazad.
\newblock Convergence analysis for column-action methods in image
  reconstruction.
\newblock {\em Numer. Algorithms}, 74(3):905--924, 2017.

\bibitem{Franchetti}
C.~Franchetti and W.~Light.
\newblock On the von {N}eumann alternating algorithm in {H}ilbert space.
\newblock {\em J. Math. Anal. Appl.}, 114(2):305--314, 1986.

\bibitem{Gearhart}
W.B. Gearhart and M.~Koshy.
\newblock Acceleration schemes for the method of alternating projections.
\newblock {\em J. Comput. Appl. Math.}, 26(3):235--249, 1989.

\bibitem{Hansen}
P.C. Hansen and J.S. J{\o}rgensen.
\newblock A{IR} {T}ools {II}: algebraic iterative reconstruction methods,
  improved implementation.
\newblock {\em Numer. Algorithms}, 79(1):107--137, 2018.

\bibitem{Jiang}
M.~Jiang and G.~Wang.
\newblock Convergence studies on iterative algorithms for image reconstruction.
\newblock {\em IEEE Trans. Med. Imaging}, 22(5):569--579, 2003.

\bibitem{Kaczmarz}
S.~Kaczmarz.
\newblock Angen\"{a}herte aufl\"{o}sung von systemen linearer gleichungen.
\newblock {\em Bulletin International de l'Acad\'{e}mie Polonaise des Sciences
  et des Lettres. Classe des Sciences Math\'{e}matiques et Naturelles.
  S\'{e}rie A, Sciences Math\'{e}matiques}, 35:355--357, 1937.

\bibitem{Strohmer}
T.~Strohmer and R.~Vershynin.
\newblock A randomized {K}aczmarz algorithm with exponential convergence.
\newblock {\em J. Fourier Anal. Appl.}, 15(2):262--278, 2009.

\end{thebibliography}

\end{document}